\documentclass[10pt]{amsart}
\usepackage{amssymb}

\newtheorem{thm}{Theorem}[section]
\newtheorem{lemma}[thm]{Lemma}
\newtheorem{proposition}[thm]{Proposition}

\newtheorem{definition}[thm]{Definition}
\newtheorem{corollary}[thm]{Corollary}

\newtheorem{question}[thm]{Question}

\newtheorem{observation}[thm]{Observation}
\newtheorem*{definition*}{Definition}

\newcommand{\p}{\mathbb{P}}
\newcommand{\q}{\mathbb{Q}}

\newcommand{\ot}{\mathrm{ot}}

\newcommand{\dom}{\mathrm{dom}}

\newcommand{\Lim}{\mathrm{Lim}}

\DeclareMathOperator{\Measuring}{\textsf{Measuring}}

\let \o = \omega

\begin{document}

\title{Parametrized Measuring and Club Guessing}

\author{David Asper\'{o} and John Krueger}

\address{David Asper\'o \\ School of Mathematics \\ University of East Anglia \\ Norwich NR4 7TJ, UK}
\email{d.aspero@uea.ac.uk}

\address{John Krueger \\ Department of Mathematics \\ 
	University of North Texas \\
	1155 Union Circle \#311430 \\
	Denton, TX 76203}
\email{jkrueger@unt.edu}

\date{August 2018; revised July 2019}

\thanks{The first author acknowledges support of EPSRC Grant EP/N032160/1. 
	The second author was partially supported by 
	the National Science Foundation Grant
	No. DMS-1464859.}

\thanks{2010 \emph{Mathematics Subject Classification:} 
	Primary 03E05, 03E35; Secondary 03E57.}

\thanks{\emph{Key words and phrases.} Strong Measuring, Club Guessing, MRP, BPFA}

\begin{abstract}
	We introduce \emph{Strong Measuring}, a maximal strengthening of J.\ T.\ Moore's 
	Measuring principle, which asserts that every collection of fewer than continuum many 
	closed bounded subsets of $\omega_1$ is measured by some club subset of $\omega_1$. 
	The consistency of Strong Measuring with the negation of \textsf{CH} is shown, 
	solving an open problem from \cite{aspero3} about parametrized measuring principles. 
	Specifically, we prove that Strong Measuring follows from \textsf{MRP} together with Martin's Axiom 
	for $\sigma$-centered forcings, as well as from \textsf{BPFA}. 
	We also consider strong versions of Measuring in the absence of the Axiom of Choice.
\end{abstract}

\maketitle

Club guessing principles at $\omega_1$ are well--studied natural weakenings of Jensen's $\diamondsuit$ principle. 
Presented in a general form, they assert the existence of a sequence 
$\vec C = \langle c_\alpha \,:\, \alpha\in \omega_1\cap \Lim\rangle$, 
where each $c_\alpha$ is a club of $\alpha$, such that $\vec C$ \emph{guesses} clubs of $\omega_1$ in some suitable sense. 
$\vec C$ guessing a club $D$ of $\omega_1$ usually means that there is some (equivalently, stationarily many) 
$\delta\in D$ such that $c_\delta\cap D$ is a suitably large subset of $c_\delta$; for example, we could require that 
$c_\delta\subseteq D$, in which case the resulting statement is called \emph{club guessing}, or that $c_\delta \cap D$ is cofinal in $\delta$, 
in which case we call the resulting statement \emph{very weak club guessing}.

Unlike the case of their versions at cardinals higher than $\omega_1$, 
for which there are non--trivial positive \textsf{ZFC} theorems (see, for example, \cite{cardinalarithmetic}), 
club guessing principles at $\omega_1$ are independent of \textsf{ZFC}. 
On the one hand, all of these principles obviously follow from $\diamondsuit$, 
and hence they hold in $L$, and they can always be forced by countably closed forcing. 
On the other hand, classical forcing axioms at the level of $\omega_1$, such as the Proper Forcing Axiom (\textsf{PFA}), 
imply the failure of even the weakest of these principles. It should nevertheless be noted that Martin's Axiom + 
$\neg \textsf{CH}$ is 
com\-pa\-tible with Club Guessing. 
This is because Martin's Axiom can always be forced by a c.c.c.\ forcing, and the fact that 
every club of $\omega_1$ in a generic extension via a c.c.c.\ forcing contains a club of $\omega_1$ from the ground model 
implies that a club--guessing sequence from the ground model remains club--guessing in the extension. 
(On the other hand, this is of course not the case for $\diamondsuit$ since the negation of \textsf{CH} violates $\diamondsuit$.)

\emph{Measuring} is a particularly strong failure of Club Guessing due to J.\ T.\ Moore (\cite{moore2}). 
Let $X$ and $Y$ be countable subsets of $\omega_1$ with the same supremum $\delta$. 
We say that 
\emph{$X$ measures $Y$} if there exists $\beta<\delta$ such that $X \setminus\beta$ is either contained in, or disjoint from, $Y$. \emph{Measuring} is the statement that for any sequence $\langle c_\alpha \,:\, \alpha\in\omega_1\cap\Lim\rangle$, where each $c_\alpha$ is a closed subset of $\alpha$, there exists a club $D\subseteq\omega_1$ such that for all limit points $\delta\in D$ 
of $D$, $D \cap \delta$ measures $c_\delta$.

Measuring can be viewed as a strong negation of Club Guessing since, as is easy to see, it implies the failure of Very Weak Club Guessing.  
Measuring follows from the Mapping Reflection Principle (\textsf{MRP}), 
and therefore from $\textsf{PFA}$, and it can be forced over any model of 
$\textsf{ZFC}$.

From Measuring as a vantage point, one can attempt to consider even stronger failures of Club Guessing. In this vein, the following parametrized family of strengthenings of Measuring was considered in \cite{aspero3}. 

\begin{definition*}\label{parametrized-measuring}
	For a cardinal $\kappa$, let $\textsf{Measuring}_{{<}\kappa}$ denote the statement that whenever $\vec{\mathcal C}=\langle \mathcal C_\alpha\,:\,\alpha\in \omega_1\cap\Lim\rangle$ is a sequence such that each $\mathcal C_\alpha$ is
	a family of fewer than $\kappa$ many closed subsets of $\alpha$, there exists a club $D\subseteq \omega_1$ with the property that for every limit point $\delta$ of $D$ and every $c\in\mathcal C_\delta$, $D\cap\delta$ measures $c$. For a cardinal $\lambda$, let $\Measuring_\lambda$ denote $\Measuring_{{<}\lambda^+}$. 
\end{definition*}

In the situation given by the above definition, we say that $D$ \emph{measures} $\vec{\mathcal C}$. 
We also define \emph{Strong Measuring} to be the statement $\Measuring_{{<}2^{\omega}}$.

In the present article we contribute to the body of information on Measuring and related strong failures of Club Guessing (see also \cite{moore2}, \cite{aspero5}, \cite{aspero}, \cite{aspero4}, and \cite{aspero3}). 
One of the questions left unresolved in \cite{aspero3} is whether $\Measuring_{\omega_1}$ is consistent at all. Answering this question was the motivation for the work in the present article. 
Our main result is that Strong Measuring + $\neg \textsf{CH}$ is consistent. In fact, this statement follows from 
	$\textsf{MRP}$ + Martin's Axiom for the class of $\sigma$-centered posets, and also from $\textsf{BPFA}$.\footnote{We can also prove the consistency of Strong Measuring with the continuum being arbitrarily large. This result will appear in a sequel to the present article.}
	We also show the failure, in $\textsf{ZFC}$, of $\Measuring_\kappa$, where $\kappa$ is among some of the classical cardinal characteristics of the continuum. Finally, we consider very strong versions of Measuring in contexts in which the Axiom of Choice fails.  

\section{Background}

We review some background material and notation which is needed for understanding the paper. 
Let $\mathfrak c$ denote the cardinality of the continuum $2^\omega$. 
A set $S \subseteq [\omega]^\omega$ is a \emph{splitting family} 
if for any infinite set $x \subseteq \omega$, there exists $A \in S$ such that $A$ \emph{splits} $x$ 
in the sense that both $x \cap A$ and $x \setminus A$ are infinite. 
The \emph{splitting number} $\mathfrak s$ is the least cardinality of some splitting family. 
Given functions $f, g : \omega \to \omega$, we say that $g$ dominates $f$ if for all $n < \omega$, 
$f(n) < g(n)$. We say that $g$ eventually dominates $f$ if there is some $m<\omega$ such that $f(n)<g(n)$ for all $n>m$.
A family $B \subseteq \omega^\omega$ is \emph{bounded} if there exists a function $g \in \omega^\omega$ 
which eventually dominates every member of $B$, and otherwise it is \emph{unbounded}. 
The \emph{bounding number} $\mathfrak b$ is the least cardinality of some unbounded family. 
Both cardinal characteristics $\mathfrak s$ and $\mathfrak b$ are uncountable.

Let $\p$ be a forcing poset. 
A set $X \subseteq \p$ is \emph{centered} if every finite subset of $X$ has a lower bound. 
We say that $\p$ is \emph{$\sigma$-centered} if it is a union of countably many centered sets. 
\emph{Martin's Axiom for $\sigma$-centered forcings} ($\textsf{MA}$($\sigma$-centered)) is the statement 
that for any $\sigma$-centered forcing $\p$ and any collection of fewer than $\mathfrak c$ many 
dense subsets of $\p$, there exists a filter on $\p$ which meets each dense set in the collection. 
More generally, let $\mathfrak{m}$($\sigma$-centered) be the least cardinality of a collection of dense 
subsets of some $\sigma$-centered forcing poset for which there does not exist a filter which meets 
each dense set in the collection. 
Note that $\textsf{MA}$($\sigma$-centered) is equivalent to the statement that 
$\mathfrak{m}$($\sigma$-centered) equals $\mathfrak{c}$.

The \emph{Bounded Proper Forcing Axiom} (\textsf{BPFA}) is the statement that whenever 
$\p$ is a proper forcing and $\langle A_i : i < \omega_1 \rangle$ is a sequence of 
maximal antichains of $\p$ each of size at most $\omega_1$, then there exists a 
filter on $\p$ which meets each $A_i$ (\cite{BPFA}). 
We note that \textsf{BPFA} implies $\mathfrak c = \omega_2$ (\cite[Section 5]{moore}). 
It easily follows that \textsf{BPFA} implies Martin's Axiom, and in particular, 
implies $\textsf{MA}$($\sigma$-centered). 
The forcing axiom \textsf{BPFA} is equivalent to the statement that for any proper forcing poset $\p$ 
and any $\Sigma_1$ statement $\Phi$ with a parameter from $H(\omega_2)$, 
if $\Phi$ holds in a generic extension by $\p$, then $\Phi$ holds in the ground model (\cite{bagaria}).

An \emph{open stationary set mapping} for an uncountable set 
$X$ and regular cardinal $\theta > \omega_1$ 
is a function $\Sigma$ whose domain is the collection of all countable elementary substructures $M$ of 
$H(\theta)$ with $X \in M$, such that for all such $M$, 
$\Sigma(M)$ is an open, $M$-stationary subset of $[X]^\omega$. 
By \emph{open} we mean in the Ellentuck topology on $[X]^\omega$, and \emph{$M$-stationary} means 
meeting every club subset of $[X]^\omega$ which is a member of $M$ 
(see \cite{moore} for the complete details). 
In this article, we are only concerned with these ideas in the simplest case that $X = \omega_1$ 
and for each $M \in \dom(\Sigma)$, $\Sigma(M) \subseteq \omega_1$. 
In this case, being open is equivalent to being open in the topology on $\omega_1$ with basis 
the collection of all open intervals of ordinals, and being $M$-stationary is equivalent to meeting every 
club subset of $\omega_1$ in $M$.

For an open stationary set mapping $\Sigma$ for $X$ and $\theta$, a \emph{$\Sigma$-reflecting sequence} is 
an $\in$-increasing and continuous sequence $\langle M_i : i < \omega_1 \rangle$ 
of countable elementary substructures of $H(\theta)$ containing $X$ as a member 
satisfying that for all limit ordinals $\delta < \omega_1$, 
there exists $\beta < \delta$ so that for all $\beta \le \xi < \delta$, 
$M_\xi \cap X \in \Sigma(M_\delta)$. 
The \emph{Mapping Reflection Principle} (\textsf{MRP}) is the statement that for any 
open stationary set mapping $\Sigma$, there exists a $\Sigma$-reflecting sequence. 
We will use the fact that for any open stationary set mapping $\Sigma$, there exists a proper forcing which adds a 
$\Sigma$-reflecting sequence (\cite[Section 3]{moore}). 
Consequently, \textsf{MRP} follows from \textsf{PFA}.

\section{Parametrized Measuring and Club Guessing}

Let $X$ and $Y$ be countable subsets of $\omega_1$ with the same supremum $\delta$. 
We say that $X$ \emph{measures} $Y$ if there exists $\beta < \delta$ 
such that $X \setminus \beta$ is either contained in, or disjoint from, $Y$. 
\emph{Measuring} is the statement that for any sequence $\langle c_\alpha : \alpha \in \omega_1 \cap \Lim \rangle$, 
where each $c_\alpha$ is a closed and cofinal subset of $\alpha$, 
there exists a club $D \subseteq \omega_1$ such that for all limit points $\alpha$ of $D$, 
$D \cap \alpha$ measures $c_\alpha$.

The next two results are due to J.\ T.\ Moore (\cite{moore2}).

\begin{thm}
	\textsf{MRP} implies Measuring.
\end{thm}

\begin{thm}
	\textsf{BPFA} implies Measuring.
\end{thm}

We now describe parametrized forms of measuring which were introduced in \cite{aspero3}. 
Let $\vec{\mathcal C} = \langle \mathcal C_\alpha : \alpha \in \omega_1 \cap \textrm{Lim} \rangle$ 
be a sequence such that each $\mathcal C_\alpha$ is a collection of 
closed and cofinal subsets of $\alpha$. 
A club $D \subseteq \omega_1$ is said to \emph{measure} $\vec{\mathcal C}$ if for all 
$\alpha \in \lim(D)$ and all $c \in \mathcal C_\alpha$, $D \cap \alpha$ measures $c$.
	
\begin{definition}
	For a cardinal $\kappa$, let \textsf{Measuring}$_{< \kappa}$ 
	denote the statement that whenever 
	$\vec{\mathcal C} = \langle \mathcal C_\alpha : \alpha \in \omega_1 \cap \textrm{Lim} \rangle$ 
	is a sequence such that each $\mathcal C_\alpha$ is a collection of fewer than $\kappa$ many 
	closed and cofinal subsets of $\alpha$, then there exists a club $D \subseteq \omega_1$ 
	which measures $\vec{\mathcal C}$. 	
	For a cardinal $\lambda$, let $\textsf{Measuring}_\lambda$ denote $\textsf{Measuring}_{< \lambda^+}$. 
	\end{definition}

Observe that the principle Measuring is the same as $\textsf{Measuring}_1$. 
If $\kappa < \lambda$, then clearly \textsf{Measuring}$_{< \lambda}$ implies 
\textsf{Measuring}$_{<\kappa}$. 
It is easy to see that \textsf{Measuring}$_{\mathfrak c}$ is false. 

\begin{definition}
	\emph{Strong Measuring} is the statement that \textsf{Measuring}$_{<\mathfrak c}$ holds.
	\end{definition}

Since the intersection of countably many clubs in $\omega_1$ is club, 
Measuring easily implies \textsf{Measuring}$_\omega$. 
In particular, Measuring together with \textsf{CH} implies Strong Measuring. 
We will prove in Section 3 the consistency of Strong Measuring together with 
$\neg \textsf{CH}$. 
We also observe at the end of that section that Measuring does not imply \textsf{Measuring}$_{\omega_1}$.

\begin{proposition}[\cite{aspero3}]
	\textsf{Measuring}$_{\mathfrak s}$ is false.
	\end{proposition}

\begin{proof}
	Fix a splitting family $S$ of cardinality $\mathfrak s$. 
	For each limit ordinal $\alpha < \omega_1$, fix a function $f_\alpha : \omega \to \alpha$ which is 
	increasing and cofinal in $\alpha$. 
	For each $A \in S$, let $c_{\alpha,A} = \bigcup \{ (f_\alpha(n),f_\alpha(n+1)] : n \in A \}$, which is 
	clearly closed and cofinal in $\alpha$. 
	Let $\mathcal C_\alpha := \{ c_{\alpha,A} : A \in S \}$.  
	Then $\vec{\mathcal C} := \langle \mathcal{C}_\alpha : \alpha \in \omega_1 \cap \Lim \rangle$ is a sequence 
	such that for each $\alpha$, $\mathcal{C}_\alpha$ is a collection of at most $\mathfrak s$ many closed 
	and cofinal subsets of $\alpha$.
	
	Let $D \subseteq \omega_1$ be a club. 
	Fix $\alpha \in \lim(D)$. 
	We will show that there exists a member of $\mathcal C_\alpha$ which $D \cap \alpha$ does not measure. 
	Define $x := \{ n < \omega : D \cap (f_\alpha(n),f_{\alpha}(n+1)] \ne \emptyset \}$. 
	Since $\alpha \in \lim(D)$, $x$ is infinite. 
	As $S$ is a splitting family, we can fix $A \in S$ which splits $x$. 
	So both $x \cap A$ and $x \setminus A$ are infinite. 
	We claim that $D \cap \alpha$ does not measure $c_{\alpha,A}$.

	Suppose for a contradiction that for some $\beta < \alpha$, 
	$(D \cap \alpha) \setminus \beta$ is either a subset of, or disjoint from, $c_{\alpha,A}$. 
	Since $A \cap x$ is infinite, we can fix $n \in A \cap x$ such that $f_\alpha(n) > \beta$. 
	Then $n \in x$ implies that $D \cap (f_\alpha(n),f_{\alpha}(n+1)] \ne \emptyset$, 
	and $n \in A$ implies that $(f_{\alpha}(n),f_{\alpha}(n+1)] \subseteq c_{\alpha,A}$. 
	It follows that $(D \cap \alpha) \setminus \beta$ meets $c_{\alpha,A}$. 
	By the choice of $\beta$, this implies that $(D \cap \alpha) \setminus \beta$ is a subset of $c_{\alpha,A}$. 
	But $x \setminus A$ is also infinite, so we can fix $m \in x \setminus A$ such that $f_{\alpha}(m) > \beta$. 
	Then $m \in x$ implies that 
	$D \cap (f_\alpha(m),f_\alpha(m+1)] \ne \emptyset$, and $m \notin A$ implies that 
	$(f_{\alpha}(m),f_{\alpha}(m+1)]$ is disjoint from $c_{\alpha,A}$. 
	Thus, there is a member of $(D \cap \alpha) \setminus \beta$ which is not in $c_{\alpha,A}$, which is a contradiction.
\end{proof}

We will prove later in this section that \textsf{Measuring}$_{\mathfrak b}$ is also false.

We now turn to parametrized club guessing. 
We recall some standard definitions. 
Consider a sequence $\vec L = \langle L_\alpha : \alpha \in \omega_1 \cap \Lim \rangle$, where each $L_\alpha$ 
is a cofinal subset of $\alpha$ with order type $\omega$ (that is, a \emph{ladder system}). 
We say that $\vec L$ is a \emph{club guessing sequence}, \emph{weak club guessing sequence}, or 
\emph{very weak club guessing sequence}, respectively, if for every club $D \subseteq \omega_1$, there exists 
a limit ordinal $\alpha < \omega_1$ such that:
\begin{enumerate}
	\item $L_\alpha \subseteq D$,
	\item $L_\alpha \setminus D$ is finite, or
	\item $L_\alpha \cap D$ is infinite, respectively.
\end{enumerate}
We say that \emph{Club Guessing}, \emph{Weak Club Guessing}, or \emph{Very Weak Club Guessing} 
holds, respectively, if there exists a club guessing sequence, a weak club guessing sequence, 
or a very weak club guessing sequence, respectively. 
It is well known that Measuring implies the failure of Very Weak Club Guessing (see Proposition 2.8 below).

\begin{definition}
	Let $\vec{\mathcal L} = \langle \mathcal L_\alpha : \alpha \in \omega_1 \cap \Lim \rangle$ 
	be a sequence where each $\mathcal L_\alpha$ is a non--empty 
	collection of cofinal subsets of $\alpha$ with order type $\omega$. 
	The sequence $\vec{\mathcal L}$ is said to be 
	a \emph{club guessing sequence}, \emph{weak club guessing sequence}, or 
	\emph{very weak club guessing sequence}, respectively, if for every club $D \subseteq \omega_1$, there exists 
	a limit ordinal $\alpha < \omega_1$ and some $L \in \mathcal L_\alpha$ such that:
	\begin{enumerate}
		\item $L \subseteq D$,
		\item $L \setminus D$ is finite, or 
		\item $L \cap D$ is infinite, respectively.
	\end{enumerate}
\end{definition}

\begin{definition}
	For a cardinal $\kappa$, let \textsf{CG}$_{< \kappa}$, 
	\textsf{WCG}$_{<\kappa}$, and \textsf{VWCG}$_{<\kappa}$, respectively, be the statements that 
	there exists a club guessing sequence, weak club guessing sequence, or very weak club guessing sequence 
	$\langle \mathcal L_\alpha : \alpha \in \omega_1 \cap \Lim \rangle$, respectively, 
	such that for each $\alpha$, $|\mathcal L_\alpha| < \kappa$. 
	Let \textsf{CG}$_{\kappa}$, 
	\textsf{WCG}$_{\kappa}$, and \textsf{VWCG}$_{\kappa}$ denote the statements 
	\textsf{CG}$_{< \kappa^+}$, 
	\textsf{WCG}$_{< \kappa^+}$, and \textsf{VWCG}$_{< \kappa^+}$, respectively.
\end{definition}

Clearly, if $\kappa < \lambda$, then 
\textsf{CG}$_{< \kappa}$ implies \textsf{CG}$_{<\lambda}$, and similarly with 
\textsf{WCG} and \textsf{VWCG}. 
Observe that Club Guessing, Weak Club Guessing, and Very Weak Club Guessing are equivalent to 
\textsf{CG}$_{1}$, \textsf{WCG}$_{1}$, and \textsf{VWCG}$_{1}$, respectively. 
Obviously, \textsf{CG}$_{\mathfrak c}$ is true. 
The weakest forms of club guessing principles which are not provable in \textsf{ZFC} 
are when the index is $< \mathfrak c$.

\begin{proposition}
	For any cardinal $\kappa \ge 2$, 
	\textsf{Measuring}$_{< \kappa}$ implies the failure of \textsf{VWCG}$_{< \kappa}$.
\end{proposition}

\begin{proof}
	Suppose for a contradiction that 
	\textsf{Measuring}$_{< \kappa}$ and \textsf{VWCG}$_{<\kappa}$ both hold. 
	Fix a very weak club guessing sequence 
	$\vec{\mathcal L} = \langle \mathcal{L}_\alpha : \alpha \in \omega_1 \cap \Lim \rangle$ 
	such that each $\mathcal L_\alpha$ has cardinality less than $\kappa$. 
	Observe that for each $\alpha$, every member of $\mathcal{L}_\alpha$ is vacuously a closed 
	subset of $\alpha$ since it has 
	order type $\omega$.

	By \textsf{Measuring}$_{< \kappa}$, there exists 
	a club $D \subseteq \omega_1$ which measures $\vec{\mathcal L}$. 
	Let $E$ be the club set of indecomposable 
	limit ordinals $\alpha>\omega$ in $\lim(D)$ 
	such that $\ot(D \cap \alpha) = \alpha$. 
	Since $\vec{\mathcal L}$ is a very weak club guessing sequence, there exists 
	a limit ordinal $\alpha$ and $L \in \mathcal L_\alpha$ such that $L \cap E$ is infinite. 
	In particular, $\alpha$ is a limit point of $E$, and hence of $D$. 

	Since $D$ measures $\vec{\mathcal L}$ and $L \in \mathcal L_\alpha$, 
	$D \cap \alpha$ measures $L$. 
	So we can fix $\beta < \alpha$ such that 
	$(D \cap \alpha) \setminus \beta$ is either a subset of, 
	or disjoint from, $L$. 
	Now $L \cap E$, and hence $L \cap D$, is infinite. 
	As $L$ has order type $\omega$, this implies that 
	$L \cap D$ is cofinal in $\alpha$. 
	By the choice of $\beta$, 
	$(D \cap \alpha) \setminus \beta$ must be a subset of $L$. 
	But since $\alpha \in E$, 
	$\ot(D \cap \alpha) = \alpha$ and $\alpha$ is 
	indecomposable, which implies that 
	$\ot((D \cap \alpha) \setminus \beta) = \alpha$. 
	As $\alpha > \omega$, this is impossible since 
	$(D \cap \alpha) \setminus \beta$ is a subset of $L$ 
	and $L$ has order type $\omega$.
\end{proof}

In particular, since Strong Measuring is consistent, so is the failure of \textsf{VWCG}$_{< \mathfrak c}$. 
(The consistency of $\neg \textsf{VWCG}_{< \mathfrak c}$ together with $\mathfrak c$ 
arbitrarily large was previously shown in \cite{aspero2}.)

\begin{proposition}[Hru\v{s}\'{a}k \cite{aspero}]
	\textsf{VWCG}$_{\mathfrak b}$ is true.
\end{proposition}	
	 
\begin{proof}
	 	Fix an unbounded family $\{ r_\alpha : \alpha < \mathfrak b \}$ in $\omega^\omega$. 
	 	For each limit ordinal $\delta < \omega_1$, fix a cofinal subset $C_\delta$ of $\delta$ 
	 	with order type $\omega$ and a bijection $h_\delta : \omega \to \delta$. 
	 	Let $C_\delta(n)$ denote the $n$-th member of 
	 	$C_\delta$ for all $n < \omega$. 
	 	For all limit ordinals $\delta < \omega_1$ and 
	 	$\alpha < \mathfrak b$, 
	 	define 
	 	$$
	 	A_\delta^\alpha := C_\delta \cup 
	 	\bigcup \{ h_\delta[r_\alpha(n)] \setminus 
	 	C_\delta(n) : n < \omega \}.
	 	$$
	 	
	 	It is easy to check that for all 
	 	$\delta$ and $\alpha$, 
	 	$A_\delta^\alpha$ has order type $\omega$ and $\sup(A_\delta^\alpha) = \delta$. 
	 	Given a club $C \subseteq \omega_1$, 
	 	let $\delta$ be a limit point of $C$ 
	 	and let $g_{C, \delta} : \omega \to \omega$ be the function given by 
	 	$$
	 	g_{C, \delta}(n) = 
	 	\min \{ m < \omega : 
	 	h_\delta(m) \in C \setminus C_\delta(n) \}.
	 	$$
	 	Now let $\alpha < \mathfrak b$ be such that $r_\alpha(n) > g_{C, \delta}(n)$ 
	 	for infinitely many $n$. 
	 	It then follows that 
	 	$|A_\delta^\alpha \cap C| = \omega$.	
	 	\end{proof}

By Propositions 2.8 and 2.9, the following is immediate.

\begin{corollary}
	\textsf{Measuring}$_{\mathfrak b}$ is false.
\end{corollary}

An obvious question is whether the parametrized versions of club guessing are actually the same as 
the usual ones. 
We conclude this section by showing that they are not.

Recall that a forcing poset $\p$ is \emph{$\omega^\omega$-bounding} if every function in 
$\omega^\omega \cap V^\p$ is dominated by a function in $\omega^\omega \cap V$.
 
\begin{lemma}[Hru\v{s}\'{a}k]
	Assume that \textsf{VWCG} fails. 
	Let $\p$ be any $\omega_1$-c.c., $\omega^\omega$-bounding forcing. 
	Then $\p$ forces that \textsf{VWCG} fails.
\end{lemma}

\begin{proof}
	Since $\p$ is $\omega_1$-c.c.\ and $\omega^\omega$-bounding, a standard argument shows that 
	whenever $p \in \p$ and $p$ forces that $\dot b \in \omega^\omega$, then there exists a function 
	$b^* \in \omega^\omega$ such that $p$ forces that $b^*$ dominates $\dot b$.
	
	Let us show that whenever $p \in \p$, 
	$\delta < \omega_1$, and $p$ forces that $\dot X$ is a cofinal subset of $\delta$ of order type $\omega$, 
	then there exists a set $Y$ with order type $\omega$ such that $p$ forces that $\dot X \subseteq Y$. 
	To see this, fix a bijection $f : \omega \to \delta$ and a strictly increasing sequence 
	$\langle \alpha_n : n < \omega \rangle$ cofinal in $\alpha$ with $\alpha_0 = 0$. 
	We claim that there exists a $\p$-name $\dot b$ for a function from $\omega$ to $\omega$ 
	such that $p$ forces that for all $n < \omega$, 
	$\dot b(n)$ is the least $m < \omega$ such that 
	$\dot X \cap [\alpha_n,\alpha_{n+1}) \subseteq f[m]$. 
	This is true since $p$ forces that $\dot X$ has order type $\omega$ and hence that  
	$\dot X \cap [\alpha_n,\alpha_{n+1})$ is finite for all $n < \omega$. 
	Fix a function $b^* : \omega \to \omega$ such that $p$ forces that $b^*$ dominates $\dot b$. 
	Now let 
	$$
	Y := \bigcup \{ f[b^*(n)] \cap [\alpha_n,\alpha_{n+1}) : n < \omega \}.
	$$
	It is easy to check that $Y$ has order type $\omega$ and $p$ forces that $\dot X \subseteq Y$.
	
	Now we are ready to prove the proposition. 
	So suppose that $p \in \p$ forces that 
	$\langle \dot X_\alpha : \alpha \in \omega_1 \cap \Lim \rangle$ is a very weak club guessing sequence. 
	By the previous paragraph, for each limit ordinal $\alpha < \omega_1$ we can fix a cofinal subset $Y_\alpha$ 
	of $\alpha$ with order type $\omega$ 
	such that $p$ forces that $\dot X_\alpha \subseteq Y_\alpha$. 
	We claim that $\langle Y_\alpha : \alpha \in \omega_1 \cap \Lim \rangle$ is a very weak club 
	guessing sequence in the ground model, which completes the proof. 
	So consider a club $C \subseteq \omega_1$. 
	Then $C$ is still a club in $V^\p$. 
	Fix $q \le p$ and a limit ordinal $\alpha < \omega_1$ such that $q$ forces that $\dot X_\alpha \cap C$ is infinite. 
	Then clearly $q$ forces that $Y_\alpha \cap C$ is infinite, so in fact, $Y_\alpha \cap C$ is infinite.
\end{proof}

\begin{proposition}
	It is consistent that $\neg \textsf{VWCG}$ and $\textsf{CG}_{\omega_1}$ both hold.
	\end{proposition}

\begin{proof}
Let $V$ be a model in which \textsf{CH} holds and \textsf{VWCG} fails. 
Such a model was shown to exist by Shelah \cite{NNR}. 
Let $\p$ be an $\omega_1$-c.c., $\omega^\omega$-bounding forcing poset 
which adds at least $\omega_2$ many reals; for example, random real forcing with product measure 
is such a forcing. 
We claim that in $V^\p$, \textsf{CG}$_{\omega_1}$ holds but \textsf{VWCG} fails. 
By Lemma 2.11, \textsf{VWCG} is false in $V^\p$. 
In $V$, define $\vec{\mathcal L} = \langle \mathcal L_\alpha : \alpha \in \omega_1 \cap \Lim \rangle$ 
by letting $\mathcal L_\alpha$ be the collection of all cofinal subsets of $\alpha$ with order type $\omega$. 
Since \textsf{CH} holds, the cardinality of each $\mathcal L_\alpha$ is $\omega_1$. 
If $C$ is a club subset of $\omega_1$ in $V^\p$, then since $\p$ is $\omega_1$-c.c., 
there is a club $D \subseteq \omega_1$ in $V$ such that $D \subseteq C$. 
In $V$, fix $d \subseteq D$ with order type $\omega$, and let $\alpha := \sup(d)$. 
Then $d \in \mathcal L_\alpha$ and $d \subseteq C$. 
Thus, $\vec{\mathcal L}$ witnesses that \textsf{CG}$_{\omega_1}$ holds in $V^\p$.
\end{proof}

\section{The Consistency of Strong Measuring and $\neg \textsf{CH}$}

As we previously mentioned, Measuring is equivalent to \textsf{Measuring}$_\omega$, and therefore under 
\textsf{CH}, Measuring is equivalent to Strong Measuring. 
In this section we establish the consistency of Strong Measuring with the negation of \textsf{CH}. 
More precisely, we will prove that \textsf{MRP} together with \textsf{MA}($\sigma$-centered) implies Strong Measuring, 
and \textsf{BPFA} implies Strong Measuring. 
Recall that both \textsf{MRP} and \textsf{BPFA} imply that $\mathfrak c = \omega_2$ (\cite{moore}).

A set $M$ is \emph{suitable} if for some regular cardinal 
$\theta > \omega_1$, $M$ is a countable elementary substructure of $H(\theta)$. 
We will follow the conventions introduced in Section 1 that the 
properties ``open'' and ``$M$-stationary'' refer to open and $M$-stationary 
subsets of $\omega_1$ (where $\omega_1$ is considered as a subspace of $[\omega_1]^\omega$).

\begin{proposition}
	Assume that $M$ is suitable. 
	Let $\delta := M \cap \omega_1$. 
	Suppose that $\mathcal Y$ is a collection of open subsets of $\delta$ such that for any 
	finite set $a \subseteq \mathcal Y$, $\bigcap a$ is $M$-stationary. 
	Then there exists a $\sigma$-centered forcing $\p$ and a collection 
	$\mathcal D$ of dense subsets of $\p$ of size at most $|\mathcal Y| + \omega$ such that whenever $G$ is a filter on 
	$\p$ in some outer model $W$ of $V$ with $\omega_1^V = \omega_1^W$ 
	which meets each member of $\mathcal D$, 
	then there exists a set $z \subseteq \delta$ in $W$ which is open, $M$-stationary, and 
	satisfies that for all $X \in \mathcal Y$, $z \setminus X$ is bounded in $\delta$.
\end{proposition}

\begin{proof}
	Define a forcing poset $\p$ to consist of conditions which are pairs 
	$(x,a)$, where $x$ is an open and bounded subset of $\delta$ in $M$ 
	and $a$ is a finite subset of $\mathcal Y$. 
	Let $(y,b) \le (x,a)$ if $y$ is an end-extension of $x$, $a \subseteq b$, and 
	$y \setminus x \subseteq \bigcap a$.
	
	Since $M$ is countable, there are only countably many possibilities for the first 
	component of a condition. 
	If $(x,a_0), \ldots, (x,a_n)$ are finitely many conditions with the same first component, 
	then easily $(x,a_0 \cup \ldots \cup a_n)$ is a condition in $\p$ which is below 
	each of the conditions $(x,a_0), \ldots, (x,a_n)$. 
	It follows that $\p$ is $\sigma$-centered.

	For each $X \in \mathcal Y$, let $D_X$ denote the set of conditions $(x,a)$ such that $X \in a$. 
	Observe that $D_X$ is dense. 
	For every club $C$ of $\omega_1$ which is a member of $M$, let $E_C$ denote the 
	set of conditions $(x,a)$ such that $x \cap C$ is non--empty. 
	We claim that $E_C$ is dense. 
	Let $(x,a)$ be a condition. 
	Since $\bigcap a$ is $M$-stationary and $\lim(C) \setminus (\sup(x)+1)$ is a club subset of 
	$\omega_1$ in $M$, we can find a limit ordinal $\alpha$ in $C \cap (\bigcap a)$ 
	which is in the interval $(\sup(x),\delta)$. 
	Since $\alpha \in \bigcap a$ and $\bigcap a$ is open, we can find $\beta < \gamma < \delta$ 
	such that $\alpha \in (\beta,\gamma) \subseteq \bigcap a$. 
	As $\sup(x) + 1 < \alpha$, without loss of generality $\sup(x) < \beta$. 
	By elementarity, the interval $b := (\beta,\gamma)$ is in $M$. 
	It follows that $(x \cup b,a)$ is a condition, $x \cup b$ end-extends $x$, 
	and $(x \cup b) \setminus x = b \subseteq \bigcap a$. 
	Thus, $(x \cup b,a) \le (x,a)$, and since $\alpha \in C$, $(x \cup b,a) \in E_C$.

	Let $\mathcal D$ denote the collection of all dense sets of the form $D_X$ where $X \in \mathcal Y$, 
	or $E_C$ where $C$ is a club subset of $\omega_1$ belonging to $M$. 
	Then $|\mathcal D| \le |\mathcal Y| + \omega$. 
	Let $G$ be a filter on $\p$ in some outer model $W$ with $\omega_1^V = \omega_1^W$ 
	which meets each dense set in $\mathcal D$. 
	Define $z := \bigcup \{ x : \exists a \ (x,a) \in G \}$. 
 	Note that since $z$ is a union of open sets, it is open (using the fact that being open 
 	is absolute between $V$ and $W$). 
 	For each club $C \subseteq \omega_1$ which lies in $M$, there exists a condition $(x,a)$ which 
 	belongs to $G \cap E_C$, and thus $x \cap C \ne \emptyset$. 
 	Therefore, $z \cap C \ne \emptyset$. 
 	Hence, $z$ is $M$-stationary.

	It remains to show that 
	for all $X \in \mathcal Y$, $z \setminus X$ is bounded in $\delta$.
 	Consider $X \in \mathcal Y$.  
 	Then we can fix $(x,a) \in G \cap D_X$, which means that $X \in a$. 
 	Now the definition of the ordering on $\p$ together with the fact that $G$ is a filter easily 
 	implies that $z \setminus x \subseteq X$. 
 	Therefore, $z \setminus X \subseteq x$, and hence $z \setminus X$ is bounded in $\delta$. 
	\end{proof}

\begin{corollary}
	Assume that $M$ is suitable. 
	Let $\delta := M \cap \omega_1$. 
	Suppose that $\mathcal Y$ is a collection of less than 
	$\mathfrak{m}$($\sigma$-centered) many open subsets of $\delta$ such that 
	for any finite set $a \subseteq \mathcal Y$, $\bigcap a$ is $M$-stationary. 
	Then there exists a set $z \subseteq \delta$ which is open, $M$-stationary, and 
	satisfies that for all $X \in \mathcal Y$, $z \setminus X$ is bounded in $\delta$.
	\end{corollary}

\begin{proof}
	Fix a $\sigma$-centered forcing $\p$ and a collection 
	$\mathcal D$ of dense subsets of $\p$ of size at most $|\mathcal Y| + \omega$ as described in Proposition 3.1. 
	Since $\mathfrak{m}$($\sigma$-centered) is uncountable, 
	$| \mathcal D | < \mathfrak{m}(\textrm{$\sigma$-centered})$. 
	Hence, there exists a filter $G$ on $\p$ which meets each dense set in $\mathcal D$. 
	By Proposition 3.1, there exists a set 
	$z \subseteq \delta$ which is open, $M$-stationary, and 
	satisfies that for all $X \in \mathcal Y$, $z \setminus X$ is bounded in $\delta$.
\end{proof}

\begin{proposition}
	Let $\vec{\mathcal C} = \langle \mathcal C_\alpha : \alpha \in \omega_1 \cap \textrm{Lim} \rangle$ 
	be a sequence such that each $\mathcal C_\alpha$ is a collection of less than 
	$\mathfrak{m}$($\sigma$-centered) many 
	closed and cofinal subsets of $\alpha$. 
	Then there exists an open stationary set mapping $\Sigma$ such that, if $W$ is any 
	outer model with the same $\omega_1$ in which there exists a $\Sigma$-reflecting sequence, then 
	there exists in $W$ a club subset of $\omega_1$ which measures $\vec{\mathcal C}$.
	\end{proposition}

\begin{proof}
	For each limit ordinal $\alpha < \omega_1$, let 
	$\mathcal D_\alpha := \{ \alpha \setminus c : c \in \mathcal C_\alpha \}$. 
	Observe that each $\mathcal D_\alpha$ is a collection of fewer than 
	$\mathfrak{m}$($\sigma$-centered) many open subsets of $\alpha$.
	
	We will define $\Sigma$ to have domain the collection of all countable 
	elementary substructures $M$ of $H(\omega_2)$. 
	Consider such an $M$ and we define $\Sigma(M)$. 
	Note that $M$ is suitable. 
	Let $\delta := M \cap \omega_1$. 
	We consider two cases. 
	In the first case, there does not exist a member of 
	$\mathcal D_\delta$ which is $M$-stationary. 
	Define $\Sigma(M) = \delta$, which is clearly open 
	and $M$-stationary.
	
	In the second case, 
	there exists some member of 
	$\mathcal D_\delta$ which is $M$-stationary. 
	A straightforward application of Zorn's lemma 
	implies that there exists a non--empty set 
	${\mathcal{Y}}_M \subseteq \mathcal D_\delta$ such that 
	for any $a \in [{\mathcal{Y}}_M]^{<\omega}$, 
	$\bigcap a$ is $M$-stationary, and moreover, 
	$\mathcal{Y}_M$ is a maximal subset of 
	$\mathcal D_\delta$ with this property. 
	Since $\mathcal{Y}_M \subseteq \mathcal D_\delta$, 
	$|\mathcal{Y}_M| < \mathfrak{m}(\textrm{$\sigma$-centered})$. 
	So the collection $\mathcal{Y}_M$ satisfies the 
	assumptions of Corollary 3.2. 
	It follows that there exists  
	a set $z_M \subseteq \delta$ which is open, 
	$M$-stationary, and satisfies that for all 
	$X \in \mathcal{Y}_M$, $z_M \setminus X$ 
	is bounded in $\delta$. 
	Now define $\Sigma(M) := z_M$.
	
	This completes the definition of $\Sigma$. 
	Consider an outer model $W$ of $V$ with the same $\omega_1$, and assume that in $W$ there exists 
	a $\Sigma$-reflecting sequence 
	$\langle M_\delta : \delta < \omega_1 \rangle$. 
	Let $\alpha_\delta := M_\delta \cap \omega_1$ 
	for all $\delta < \omega_1$. 
	Let $D$ be the club set of $\delta < \omega_1$ such that $\alpha_\delta = \delta$. 
	We claim that $D$ measures $\vec{\mathcal C}$.

	Consider $\delta \in \lim(D)$. 
	Then $\delta = \alpha_\delta = M_\delta \cap \omega_1$. 
	Let $M := M_\delta$. 
	We first claim that if $c \in \mathcal C_\delta$ and $\delta \setminus c$ is not $M$-stationary, 
	then for some $\beta < \delta$, 
	$(D \cap \delta) \setminus \beta \subseteq c$. 
	Fix a club subset $E$ of $\omega_1$ in $M$ which is disjoint from $\delta \setminus c$. 
	By the continuity of the $\Sigma$-reflecting sequence, there exists 
	$\beta < \delta$ such that $E \in M_\beta$. 
	We claim that $(D \cap \delta) \setminus \beta \subseteq c$. 
	Let $\xi \in (D \cap \delta) \setminus \beta$. 
	Then $E \in M_\xi$, and hence by elementarity, 
	$\xi = M_\xi \cap \omega_1 \in E$. 
	Since $E$ is disjoint from $\delta \setminus c$, $\xi \in c$.
		
	We split the argument according to the two cases in the definition of $\Sigma(M)$. 
	In the first case, there does not exist a member of $\mathcal D_\delta$ 
	which is $M$-stationary. 
	Consider $c \in \mathcal C_\delta$. 
	Then $\delta \setminus c$ is not $M$-stationary. 
	By the previous paragraph, 
	there exists $\beta < \delta$ such that $(D \cap \delta) \setminus \beta \subseteq c$.

	In the second case, there exists a member of $\mathcal D_\delta$ which is $M$-stationary. 
	Consider $c \in \mathcal C_\delta$. 
	Then $X := \delta \setminus c \in \mathcal D_\delta$. 
	We consider two possibilities. 
	First, assume that $X$ is in $\mathcal{Y}_M$. 
	By the choice of $\mathcal{Y}_M$ and $z_M$, we know that 
	$z_M \setminus X$ is bounded in $\delta$. 
	So fix $\beta_0 < \delta$ so that $z_M \setminus \beta_0 \subseteq X$. 
	By the definition of being a $\Sigma$-reflecting sequence, 
	there exists $\beta_1 < \delta$ so that for all $\beta_1 \le \xi < \delta$, 
	$M_\xi \cap \omega_1 \in \Sigma(M) = z_M$. 
	Let $\beta := \max \{ \beta_1, \beta_2 \}$. 
	Consider $\xi \in (D \cap \delta) \setminus \beta$. 
	Then $\xi \ge \beta_1$ implies that $\xi = M_\xi \cap \omega_1 \in z_M$. 
	So $\xi \in z_M \setminus \beta_0 \subseteq X = \delta \setminus c$.

	Secondly, assume that $X$ is not in $\mathcal{Y}_M$. 
	By the maximality of $\mathcal{Y}_M$, there exists a set $a \in [\mathcal{Y}_M]^{<\omega}$ 
	such that $X \cap \bigcap a$ is not $M$-stationary. 
	Fix a club $E$ in $M$ which is disjoint from 
	$X \cap \bigcap a$. 
	By the continuity of the $\Sigma$-reflecting sequence, there exists 
	$\beta < \delta$ such that $E \in M_\beta$. 
	Consider $\xi \in (D \cap \delta) \setminus \beta$. 
	Then $E \in M_\xi$, which implies that $\xi = M_\xi \cap \omega_1 \in E$. 
	Thus, $\xi$ is not in $X \cap \bigcap a$. 
	On the other hand, letting $a = \{ X_0, \ldots, X_n \}$, 
	for each $i \le n$ the previous paragraph implies that there exists  
	$\beta_i < \delta$ such that 
	$(D \cap \delta) \setminus \beta_i \subseteq X_i$. 
	Let $\beta^*$ be an ordinal in $\delta$ which is 
	larger than $\beta$ and $\beta_i$ for all $i \le n$. 
	Consider $\xi \in (D \cap \delta) \setminus \beta^*$. 
	Then by the choice of $\beta$, 
	$\xi \notin X \cap \bigcap a$. 
	By the choice of the $\beta_i$'s, $\xi \in \bigcap a$. 
	Therefore, $\xi \notin X = \delta \setminus c$, which means that 
	$\xi \in c$. 
	Thus, $(D \cap \delta) \setminus \beta^* \subseteq c$.
	\end{proof}

\begin{corollary}
	Assume \textsf{MRP} and \textsf{MA}($\sigma$-centered). 
	Then Strong Measuring holds.
	\end{corollary}

\begin{proof}
	Let $\vec{\mathcal C} = \langle \mathcal C_\alpha : 
	\alpha \in \omega_1 \cap \textrm{Lim} \rangle$ 
	be a sequence such that each $\mathcal C_\alpha$ 
	is a collection of fewer than $\mathfrak c$ 
	many closed and cofinal subsets of $\alpha$. 
	We claim that there exists a club subset of $\omega_1$ which measures $\vec{\mathcal C}$. 
	By 
	\textsf{MA}($\sigma$-centered), 
	$\mathfrak{m}$($\sigma$-centered) equals $\mathfrak c$. 
	So each $\mathcal C_\alpha$ has size less 
	than $\mathfrak{m}$($\sigma$-centered).
	
	By Proposition 3.3, 
	there exists an open stationary set mapping 
	$\Sigma$ such that, if $W$ is any 
	outer model with the same $\omega_1$ in which there exists a $\Sigma$-reflecting sequence, then 
	there exists in $W$ a club subset of $\omega_1$ which measures $\vec{\mathcal C}$. 
	Applying \textsf{MRP}, there exists 
	a $\Sigma$-reflecting sequence in $V$. 
	Thus, in $V$ there exists a club subset of 
	$\omega_1$ which measures $\vec{\mathcal C}$.
	\end{proof}
 
\begin{corollary}
	Assume \textsf{BPFA}. 
	Then Strong Measuring holds.
	\end{corollary}

\begin{proof}
	Let $\vec{\mathcal C} = \langle \mathcal C_\alpha : 
	\alpha \in \omega_1 \cap \textrm{Lim} \rangle$ 
	be a sequence such that each $\mathcal C_\alpha$ 
	is a collection of fewer than $\mathfrak c = \omega_2$ 
	many closed and cofinal subsets of $\alpha$. 
	We claim that there exists a 
	club subset of $\omega_1$ which measures 
	$\vec{\mathcal C}$. 
	Since $\mathfrak c = \omega_2$, $\vec{\mathcal C}$ is a member of $H(\omega_2)$. 
	Thus, the existence of a club subset of $\omega_2$ which measures $\vec{\mathcal C}$ 
	is expressible as a $\Sigma_1$ statement involving a parameter in $H(\omega_2)$. 
	By \textsf{BPFA}, it suffices to show that there 
	exists a proper forcing which forces that such a club exists.
	
	Now \textsf{BPFA} implies Martin's Axiom, and in particular, that 
	$\mathfrak{m}$($\sigma$-centered) is equal to $\mathfrak c$. 
	So each $\mathcal C_\alpha$ has size less 
	than $\mathfrak{m}$($\sigma$-centered). 
	By Proposition 3.3, 
	there exists an open stationary set mapping 
	$\Sigma$ such that, if $W$ is any 
	outer model with the same $\omega_1$ in which there exists a $\Sigma$-reflecting sequence, then 
	there exists in $W$ a club subset of $\omega_1$ which measures $\vec{\mathcal C}$. 
	By \cite[Section 3]{moore}, there exists a proper forcing $\p$ which adds 
	a $\Sigma$-reflecting sequence, so in $V^\p$ there is a club subset of $\omega_1$ 
	which measures $\vec{\mathcal C}$. 
\end{proof}

We now sketch a proof that \textsf{MRP} alone does not imply Strong Measuring. 
In particular, Measuring does not imply Strong Measuring. 
Start with a model of \textsf{CH} in which there exists a supercompact cardinal $\kappa$. 
Construct a forcing iteration $\p$ in the standard way to obtain a model of \textsf{MRP}. 
To do this, fix a Laver function $f : \kappa \to V_\kappa$. 
Then define a countable support forcing iteration 
$\langle \p_\alpha, \dot \q_\beta : \alpha \le \kappa, \beta < \kappa \rangle$ as follows. 
Given $\p_\alpha$, consider $f(\alpha)$. 
If $f(\alpha)$ happens to be a $\p_\alpha$-name for some open stationary set mapping, then let 
$\dot \q_\alpha$ be a $\p_\alpha$-name for a proper forcing which adds an $f(\alpha)$-reflecting sequence. 
Otherwise let $\dot \q_\alpha$ be a $\p_\alpha$-name for $Col(\omega_1,\omega_2)$. 
Now define $\p := \p_\kappa$. 
Arguments similar to those in the standard construction of a model of 
\textsf{PFA} can be used to show that $\p$ forces \textsf{MRP}.

The forcing for adding a $\Sigma$-reflecting sequence for a given open stationary set mapping 
does not add reals (\cite[Section 3]{moore}). 
In particular, it is vacuously $\omega^\omega$-bounding. 
The property of being proper and $\omega^\omega$-bounding is preserved under countable 
support forcing iterations (\cite[Theorem 3.5]{abraham}), so $\p$ is also $\omega^\omega$-bounding. 
In particular, $V \cap \omega^\omega$ is an unbounded family in $V^\p$, and it has size $\omega_1$ 
since \textsf{CH} holds in $V$. 
It follows that the bounding number $\mathfrak b$ is equal to $\omega_1$. 
But by Corollary 2.9, \textsf{Measuring}$_{\mathfrak b}$ is false. 
So $\p$ forces that $\textsf{Measuring}_{\omega_1}$ is false. 
As $\mathfrak c = \omega_2$ in $V^\p$, Strong Measuring fails in $V^\p$.

We also note that Strong Measuring plus $\mathfrak c = \omega_2$ is consistent with 
the existence of an $\omega_1$-Suslin tree. 
Namely, both the forcing for adding a $\Sigma$-reflecting sequence for a given 
open stationary set mapping $\Sigma$, as well as any $\sigma$-centered forcing, 
preserve Suslin trees (\cite{miyamoto1}). 
And the property of being proper and preserving a 
given Suslin tree is preserved under countable support forcings iterations (\cite{miyamoto2}). 
So starting with a model in which there exists an $\omega_1$-Suslin tree $S$ and a supercompact cardinal $\kappa$, 
we can iterate forcing similar to the argument in the preceding paragraphs to produce 
a model of \textsf{MA}($\sigma$-centered) plus \textsf{MRP} in which $S$ is an 
$\omega_1$-Suslin tree. 
By Corollary 3.4, Strong Measuring holds in that model.

\section{Measuring Without the Axiom of Choice}

Another natural way to strengthen $\Measuring$ is to allow, in the sequence to be measured, not just closed sets, but also sets of higher Borel complexity. This line of strengthenings of $\Measuring$ was also con\-si\-dered in \cite{aspero3}. For completeness, we are including here the correspon\-ding observations.

The version of $\Measuring$ where one considers sequences $\vec X=\langle X_\alpha\,:\,\alpha\in\omega_1\cap \Lim\rangle$, with each $X_\alpha$ an open subset of $\alpha$ in the order topology, is of course equivalent to $\Measuring$. A natural next step would therefore be to consider sequences in which each $X_\alpha$ is a countable union of closed sets. This is obviously the same as allowing each $X_\alpha$ to be an arbitrary subset of $\alpha$. Let us call the corresponding statement $\Measuring^\ast$:

\begin{definition} $\Measuring^\ast$ holds if and only if for every sequence $\vec X=\langle X_\alpha\,:\,\alpha\in\omega_1\cap \Lim\rangle$, if $X_\alpha\subseteq\alpha$ for each $\alpha$, then there is some club $D\subseteq\o_1$ such that for every limit point $\delta\in D$ of $D$, $D\cap\delta$ measures $X_\delta$.
\end{definition}

It is easy to see that 
$\Measuring^\ast$  is false in $\textsf{ZFC}$.
In fact, given a stationary and co-stationary $S\subseteq\omega_1$, there is no club of $\omega_1$ measuring $\vec X=\langle S\cap\alpha\,:\,\alpha\in\omega_1\cap \Lim\rangle$. The reason is that if $D$ is any club of $\omega_1$, then both $D\cap S\cap\delta$ and $(D\cap\delta)\setminus S$ are cofinal subsets of $\delta$ for each $\delta$ in the club of limit points in $\omega_1$ of both $D\cap S$ and $D\setminus S$. 

The status of $\Measuring^\ast$ is more interesting in the absence of the Axiom of Choice. Let $\mathcal C_{\omega_1}=\{X\subseteq\omega_1\,:\,C \subseteq X\mbox{ for some club $C$ of $\omega_1$}\}$.

\begin{observation}  ($\textsf{ZF}$+ $\mathcal C_{\omega_1}$ is a normal filter on $\omega_1$) 
	Suppose $\vec X=\langle X_\delta\,:\,\delta\in\omega_1\cap \Lim\rangle$ is such that
	
	\begin{enumerate}
		\item  $X_\delta\subseteq\delta$ for each $\delta$.
		\item For each club $C\subseteq\omega_1$,
		
		\begin{enumerate}
			
			\item there is some $\delta\in C$ such that $C\cap X_\delta\neq\emptyset$, and
			
			\item there is some $\delta\in C$ such that $(C\cap\delta)\setminus X_\delta\neq\emptyset$.
			
		\end{enumerate}
		
	\end{enumerate}
	
	Then there is a stationary and co-stationary subset of $\omega_1$ definable from $\vec X$.
	
\end{observation}

\begin{proof}
	We have two possible cases. The first case is that in which for all $\alpha<\omega_1$, either
	
	\begin{itemize}
		
		\item  $W_\alpha^0=\{\delta<\omega_1\,:\,\alpha\notin X_\delta\}$ is in $\mathcal C_{\omega_1}$, or
		
		\item $W_\alpha^1=\{\delta<\omega_1\,:\,\alpha\in X_\delta\}$ is in $\mathcal C_{\omega_1}$.
		
	\end{itemize}

	For each $\alpha<\omega_1$, let $W_\alpha$ be $W_\alpha^\epsilon$ for the unique $\epsilon\in\{0, 1\}$ such that $W_\alpha^\epsilon\in \mathcal C_{\omega_1}$, and let $W^\ast=\Delta_{\alpha<\omega_1} W_\alpha\in\mathcal C_{\omega_1}$. Then $X_{\delta_0}=X_{\delta_1}\cap\delta_0$ for all $\delta_0<\delta_1$ in $W^\ast$. It then follows, by (2), that $S=\bigcup_{\delta\in W^\ast}X_\delta$, which of course is definable from $\vec C$, is a stationary and co-stationary subset of $\omega_1$. Indeed, suppose $C\subseteq\omega_1$ is a club, and let us fix a club $D\subseteq W^\ast$. There is then some $\delta\in C\cap D$ and some $\alpha\in C\cap D\cap X_\delta$. But then $\alpha\in S$ since $\delta\in W^\ast$ and $\alpha\in W^\ast\cap X_\delta$.  There is also some $\delta\in C\cap D$ and some $\alpha\in C\cap D$ such that $\alpha\notin X_\delta$, which implies that $\alpha\notin S$ by a symmetrical argument, using the fact that $X_{\delta_0}=X_{\delta_1}\cap\delta_0$ for all $\delta_0<\delta_1$ in $W^\ast$.
	
	The second possible case is that there is some $\alpha<\omega_1$ with the property that both $W^0_\alpha$ and $W^1_\alpha$ are stationary subsets of $\omega_1$. But now we can let $S$ be $W^0_\alpha$, where $\alpha$ is first such that $W^0_\alpha$ is stationary and co-stationary. 
\end{proof}

It is worth comparing the above observation with Solovay's classic result that an $\omega_1$--sequence of pairwise disjoint stationary subsets of $\omega_1$ is definable from any given ladder system on $\omega_1$ (working in the same theory).

\begin{corollary}\label{measuring-ast} ($\textsf{ZF}$+ $\mathcal C_{\omega_1}$ is a normal filter on $\omega_1$) The following are equivalent.
	
	\begin{enumerate}
		
		\item $\mathcal C_{\omega_1}$ is an ultrafilter on $\omega_1$;
		
		\item $\Measuring^\ast$;

		\item 
		For every sequence 
		$\langle X_\alpha\,:\,\alpha\in\omega_1\cap\Lim \rangle$, if $X_\alpha\subseteq\alpha$ for each $\alpha$, 
		then there is a club $C\subseteq\omega_1$ such that either 
		
		\begin{itemize}
			
			\item $C\cap\delta\subseteq X_\delta$ for every $\delta\in C$, or
			\item $C\cap X_\delta=\emptyset$ for every $\delta\in C$.
		\end{itemize}
		 
	\end{enumerate}
	
\end{corollary}

\begin{proof} (3) trivially implies (2), and by the observation (1) implies (3). Finally, to see that (2) implies (1), note that the argument right after the definition of $\Measuring^\ast$ uses only $\textsf{ZF}$ together with the regularity of 
	$\omega_1$ and the negation of (1). 
\end{proof}

In particular, the strong form of $\Measuring^\ast$ given by (3) in the above observation follows from 
$\textsf{ZF}$ together with the Axiom of Determinacy. 

We finish this digression into set theory without the Axiom of Choice 
by observing that any attempt to parametrize $\Measuring^\ast$, in the same vein as we did with 
\textsf{Measuring}, gives rise to principles vacuously equivalent to $\Measuring^\ast$ itself, at least when the parametrization is done with the alephs.\footnote{This was pointed out by Asaf Karagila.} 

Specifically, given an aleph $\kappa$, let us define $\Measuring^*_\kappa$ as the statement that for every sequence $\langle\mathcal{X}_\alpha\,:\,\alpha\in\omega_1\cap\Lim\rangle$, if each $\mathcal{X}_\alpha$ is a set of cardinality at most $\kappa$ consisting of subsets of $\alpha$, then there is a club $D\subseteq \omega_1$ such that for every limit point $\delta\in D$ of $D$, $D\cap\delta$ measures $X$ for all $X\in\mathcal{X}_\delta$. Then $\Measuring^\ast_{\omega}$ is clearly equivalent to $\Measuring^\ast$ under 
$\textsf{ZF}$ together with the normality of $\mathcal C_{\omega_1}$ and the Axiom of Choice for countable families of subsets of $\omega_1$ (which of course follows from the Axiom of Choice for countable families of sets of reals, and therefore also from $\textsf{ZF} + \textsf{AD}$). On the other hand, working in 
$\textsf{ZF}$ + $\mathcal C_{\omega_1}$ is a normal filter on $\omega_1$, we have that $\Measuring^\ast_{\omega_1}$ follows vacuously from $\Measuring^\ast$ simply because under $\Measuring^\ast$ there is no sequence $\langle\mathcal{X}_\alpha\,:\,\alpha\in\omega_1\cap\Lim\rangle$ as in the definition of 
$\Measuring^\ast_{\omega_1}$ and such that $\arrowvert \mathcal{X}_\alpha\arrowvert=\omega_1$ for some $\alpha$; indeed, $\Measuring^\ast$ implies, over this base theory, that $\mathcal C_{\omega_1}$ is an ultrafilter (Corollary \ref{measuring-ast}), and if $\mathcal C_{\omega_1}$ is an ultrafilter then there is no $\omega_1$-sequence of distinct reals, whereas the existence of a family of size $\omega_1$ consisting of subsets of some fixed countable ordinal clearly implies that there is such a sequence.

\bigskip

We conclude the article with two natural questions.

\begin{question}
	Is \textsf{Measuring}$_{\mathfrak p}$ false?
\end{question}

\begin{question} Are Measuring and Strong Measuring equivalent statements assu\-ming Martin's Axiom?
\end{question}

\bibliographystyle{plain}
\bibliography{measuring}

\begin{thebibliography}{10}

\bibitem{abraham}
U.~Abraham.
\newblock Proper forcing.
\newblock In {\em Handbook of set theory. Vols. 1, 2, 3}, pages 333--394.
  Springer, Dordrecht, 2010.

\bibitem{aspero3}
D.~Asper\'{o} and M.A. Mota.
\newblock Few new reals.
\newblock Preprint.

\bibitem{aspero5}
D.~Asper\'{o} and M.A. Mota.
\newblock Forcing consequences of {PFA} together with the continuum large.
\newblock {\em Trans. Amer. Math. Soc.}, 367:6103--6129, 2015.

\bibitem{aspero2}
D.~Asper\'{o} and M.A. Mota.
\newblock A generalization of {M}artin's axiom.
\newblock {\em Israel J. Math.}, 210(1):193--231, 2015.

\bibitem{aspero}
D.~Asper\'{o} and M.A. Mota.
\newblock Separating club-guessing principles in the presence of fat forcing
  axioms.
\newblock {\em Ann. Pure Appl. Logic}, 167(3):284--308, 2016.

\bibitem{aspero4}
D.~Asper\'{o} and M.A. Mota.
\newblock Measuring club-sequences with the continuum large.
\newblock {\em J. Symbolic Logic}, 82(3):1066--1079, 2017.

\bibitem{bagaria}
J.~Bagaria.
\newblock Bounded proper forcing axioms as principles of generic absoluteness.
\newblock {\em Arch. Math. Logic}, 39(6):393--401, 2000.

\bibitem{moore2}
T.~Eisworth, D.~Milovich, and J.~Moore.
\newblock Iterated forcing and the continuum hypothesis.
\newblock In J.~Cummings and E.~Schimmerling, editors, {\em Appalachian set
  theory 2006-2012}, London Math. Soc. Lecture Notes series, pages 207--244.
  Cambridge Univ. Press, 2013.

\bibitem{BPFA}
M.~Goldstern and S.~Shelah.
\newblock The bounded proper forcing axiom.
\newblock {\em J. Symbolic Logic}, 60(1):58--73, 1995.

\bibitem{miyamoto2}
T.~Miyamoto.
\newblock $\omega_1$-{S}ouslin trees under countable support iterations.
\newblock {\em Fund. Math.}, 142(3):257--261, 1993.

\bibitem{miyamoto1}
T.~Miyamoto and T.~Yorioka.
\newblock Some results in the extension with a coherent {S}uslin tree, part
  {II} {(RIMS 2012)}.
\newblock In {\em RIMS Kokyuroku}, number 1851, pages 49--61, 2013.

\bibitem{moore}
J.~T. Moore.
\newblock Set mapping reflection.
\newblock {\em J. Math. Log.}, 5(1):87--97, 2005.

\bibitem{NNR}
S.~Shelah.
\newblock {NNR} revisited.
\newblock Preprint.

\bibitem{cardinalarithmetic}
S.~Shelah.
\newblock {\em Cardinal Arithmetic}.
\newblock Oxford Logic Guides, 29. Oxford Science Publications. The Clarendon
  Press, Oxford University Press, New York, 1994.

\end{thebibliography}

\end{document}